\documentclass[12pt,a4paper,reqno]{amsart}
\usepackage{a4wide}
\usepackage{amsmath}
\usepackage{amsfonts}
\usepackage{amsthm}
\usepackage{amssymb}

\newtheorem{theorem}{Theorem}
\newtheorem{lemma}{Lemma}
\newtheorem{proposition}{Proposition}
\newtheorem{remark}[]{Remark}
\newtheorem{corollary}{Corollary}

\newtheorem{conjecture}[]{Conjecture}

\numberwithin{equation}{section}

\begin{document}
\title{The dynamics of gonosomal evolution operators}
\author{Akmal T.~Absalamov}
\address{ Samarkand State University,
Boulevard str., 140104,
Samarkand, Uzbekistan.}
\email{absalamov@gmail.com}
\author{Utkir A.~Rozikov}
\address{Institute of Mathematics, 81,
Mirzo Ulug'bek str., 100170,
Tashkent, Uzbekistan.}
\email{rozikovu@yandex.ru}

\date{}

\maketitle

\begin{abstract}
In this paper we investigate the dynamical
systems generated by gonosomal evolution operator of sex linked
inheritance depending on parameters. Mainly we study dynamical
systems of a hemophilia which is biological group of
disorders connected with genes that diminish the body's ability
to control blood clotting or coagulation that is used to stop
bleeding when a blood vessel is broken. For the gonosomal operator
we discrebe all forms and give explicitly the types of fixed points.
Moreover we study limit points of the trajectories of the
corresponding dynamical system.
\end{abstract}

%-----------------------------------------------------------------------------
%-----------------------------------------------------------------------------
%                           INTRODUCTION
%-----------------------------------------------------------------------------
%-----------------------------------------------------------------------------
\section{Introduction}\label{int}
\noindent In biology sex is determined genetically: males and
females have different genes that specify their sexual morphology.
In animals this is often accompanied by chromosomal differences.
There are some sex linked systems which depend on temperature and
even some of systems have sex change phenomenon, see
\cite{Rozikov.book.2013} for more details. For mathematical models
of bisexual population, see \cite{Karlin.book.1984},
\cite{Ladra.book.2013}, \cite{Lyubich.book.1992} and
\cite{Reed.book.1997}. In \cite{Varro.book.2015} an algebra
associated to a sex change is constructed.

In this paper we consider evolution of a hemophilia which is a
lethal recessive $X$-linked disorder: a female carrying two
alleles for hemophilia die. Therefore, if we denote by $X^h$ the
gonosome  $X$ carrying the hemophilia, there are only two female
genotypes: $XX$ and $XX^h$ ($X^hX^h$ is lethal) and two male
genotypes: $XY$ and $X^hY$. We have four types of crosses defined
as
\begin{align*}
&XX   \times{XY}   \rightarrow{{a_1}XX},\: a_2XY, \\
&XX   \times{X^hY} \rightarrow{c_1XX^h},\: c_2XY, \\
&XX^h \times{XY}   \rightarrow{b_1XX},\: b_2XX^h,\: b_3XY, b_4X^hY, \\
&XX^h \times{X^hY} \rightarrow{d_1XX^h}, \: d_2XY, \: d_3X^hY.
\end{align*}

Let $F=\{XX,XX^h\}$ and $M=\{XY,X^h Y\}$ be sets of genotypes.
Assume that state of the set $F$ is given by a real vector $(x,y)$
and state of $M$ by a real vector $(u,v)$. Then a state of the set
$F\cup{M}$ is given by the vector $t=(x,y,u,v)\in{\mathbb{R}^4}$.
If $t'=(x',y',u',v')$ is a state of the system $F\cup{M}$ in the
next generation, then by the above rule we get the evolution
operator $W:\mathbb{R}^4\rightarrow{\mathbb{R}^4}$ defined by
\begin{equation}\label{A1}
W: \left\{
\begin{array}{ll}
x'=a_1xu+b_1yu,\vspace{1,5mm} \cr
y'=c_1xv+b_2yu+d_1yv,\vspace{1,5mm} \cr
u'=a_2xu+c_2xv+b_3yu+d_2yv,\vspace{1,5mm} \cr v'=b_4yu+d_3yv.
\end{array}
 \right.
\end{equation}
This example can be generalized as follows. Suppose that the set
of female types is $F=\{1,2,...,\eta\}$ and the set of male types
is $M=\{1,2,...,\nu\}$. Let $x=(x_1,x_2,...,x_\eta)\in
\mathbb{R}^\eta$ be a state of $F$ and $y=(y_1,y_2,...,y_{\nu})\in
\mathbb{R}^{\nu}$ be a state of $M$. Consider  $p_{ir,j}^{(f)}$
and $p_{ir,l}^{(m)}$ as some inheritance non-negative real
coefficients (not necessarily probabilities) with
$$\sum_{j=1}^{\eta}p_{ir,j}^{(f)}
+\sum_{l=1}^{\nu}p_{ir,l}^{(m)}=1$$ and the corresponding
evolution operator
\begin{equation}\label{A2}
W: \left\{
\begin{array}{ll}
x'_j=\sum_{i,r=1}^{\eta,\nu}p_{ir,j}^{(f)}x_iy_r, \quad
j=1,...,n\vspace{1,5mm}\cr
y'_l=\sum_{i,r=1}^{\eta,\nu}p_{ir,l}^{(m)}x_iy_r, \quad
l=1,...,\nu.
\end{array}
 \right.
\end{equation}
This operator is called gonosomal evolution operator.

The main problem for a given discrete-time dynamical system is to
describe the limit points of the trajectory
$\{t^{(n)}\}_{n=0}^{\infty}$ for arbitrarily given
$t^{(0)}=(x,y)\in{\mathbb{R}^{\eta+\nu}}$, where
$$t^{(n)}=W^{n}(t)=\underbrace {W(W(...W(t^{(0)}))...)}_{n}$$
denotes the $n$ times iteration of $W$ to $t^{(0)}$.

Note that the operator \eqref{A2} describes evolution of a
hemophilia. The dynamical system generated by the gonosomal
operator \eqref{A2} is complicated. In this paper we study the
dynamical system generated by the gonosomal operator \eqref{A1}
which is a particular case of \eqref{A2} corresponding to the case
$\eta=\nu=2$ and the coefficients
\begin{equation}
\begin{aligned}
&p_{11,1}^{(f)} =a_1,  &p_{11,2}^{(f)} &=0,    &p_{11,1}^{(m)} &=a_2,  &p_{11,2}^{(m)}&=0,\\
&p_{12,1}^{(f)} =0,    &p_{12,2}^{(f)} &=c_1,  &p_{12,1}^{(m)} &=c_2,  &p_{12,2}^{(m)}&=0,\\
&p_{21,1}^{(f)} =b_1,  &p_{21,2}^{(f)} &=b_2,  &p_{21,1}^{(m)} &=b_3,  &p_{21,2}^{(m)}&=b_4,\\
&p_{22,1}^{(f)} =0,    &p_{22,2}^{(f)} &=d_1,  &p_{22,1}^{(m)}
&=d_2,  &p_{22,2}^{(m)}&=d_3,
\end{aligned}
\end{equation}
where $a_1$, $a_2$, $c_1$, $c_2$, $b_1$, $b_2$, $b_3$, $b_4$,
$d_1$, $d_2$, $d_3$ are non-negative real numbers such that
\begin{equation}
a_1+a_2=c_1+c_2=b_1+b_2+b_3+b_4=d_1+d_2+d_3=1.
\end{equation}
\begin{remark}
An analogy of this problem was discussed in
\cite{Rozikov.book.2016} for the classical case when
$$a_1=a_2=c_1=c_2=\frac{1}{2}, \ \ b_1=b_2=b_3=b_4=\frac{1}{4}, \ \ d_1=d_2=d_3=\frac{1}{3}.$$
\end{remark}
%========================================================
\section{The types of the fixed points.}\label{S1}
\noindent A point $s$ is called a fixed point of the operator $W$
if $s=W(s)$. Let us find all the forms of the fixed points of $W$
given by \eqref{A1}, i.e.~we solve the following system of
equations for $(x,y,u,v)$
\begin{equation}\label{A3}
\left\{
\begin{array}{ll}
x=a_1xu+b_1yu,\vspace{1,5mm} \cr
y=c_1xv+b_2yu+d_1yv,\vspace{1,5mm} \cr
u=a_2xu+c_2xv+b_3yu+d_2yv,\vspace{1,5mm} \cr v=b_4yu+d_3yv.
\end{array}
\right.
\end{equation}
It is easy to see that $s_{1}=(0,0,0,0)$ is a solution of the
system \eqref{A3}. If $u=0$, then from the first equation we get
$x=0$. If $y=0$, then from the last equation we get $v=0$.
Moreover, if $x=y=0$, then the third and the last equations yield
$u=v=0$. If $u=v=0$, then the first and the second equations give
$x=y=0$. That is why the fixed points of the operator \eqref{A1}
might be of the following forms:
\begin{equation*}
\begin{aligned}
\text{I})\, &s_1=(0,0,0,0),  \qquad \text{II})\, s_2=(x,0,u,0), \qquad \text{III})\, s_3=(0,y,0,v),\\
\text{IV})\,  &s_4=(0,y,u,0), \qquad \text{V}) \, s_5=(x,y,u,v),
\:\:\text{where} \:\: xyuv\neq{0}.
\end{aligned}
\end{equation*}

\begin{remark}
For the given operator \eqref{A1} the forms {\upshape{II), III),
IV)}} of the fixed points are uniquely defined. Indeed the system
of equations \eqref{A3} gives us the following:
\begin{equation*}
\begin{aligned}
(x,0,u,0) &=\Bigl(\frac{1}{a_2},0,\frac{1}{a_1},0 \Bigr)  \quad &\text{when}& \quad a_1,a_2\in(0,1), \\
(0,y,0,v) &=\Bigl(0,\frac{1}{d_3},0,\frac{1}{d_1}\Bigr)   \quad &\text{when}& \quad d_2=0 \;\, \text{and} \;\, d_1,d_3\in(0,1), \\
(0,y,u,0) &=\Bigl(0,\frac{1}{b_3},\frac{1}{b_2},0\Bigr)   \quad &\text{when}& \quad b_1=b_4=0 \;\, \text{and} \;\, b_2,b_3\in(0,1).\\
\end{aligned}
\end{equation*}
\end{remark}
\begin{remark}
For the given operator \eqref{A1} the form {\upshape{V)}} of the
fixed points might not be defined uniquely. To see this, consider
the evolution operator
\begin{equation*}
W_0: \left\{
\begin{aligned}
\begin{array}{ll}
x'&=\frac{1}{2}xu, \vspace{1,5mm} \cr
y'&=\frac{1}{2}yu,\vspace{1,5mm} \cr
u'&=\frac{1}{2}xu+xv+\frac{1}{2}yu+\frac{1}{2}yv,\vspace{1,5mm}\cr
v'&=\frac{1}{2}yv.
\end{array}
\end{aligned}
\right.
\end{equation*}
One can check that $s=(1,2,2,-\frac{1}{2})$ and
$s=(2,2,2,-\frac{2}{3})$ are fixed points of the operator $W_0$
which are of the form V).
\end{remark}
In order to find the type of the fixed points of the operator
\eqref{A1} we consider the Jacobi matrix

$$J(s)=J_W=\left(%
\begin{array}{cccc}
  a_1u      & b_1u      & a_1x+b_1y & 0 \\
  c_1v      & b_2u+d_1v & b_2y      & c_1x+d_1y \\
  a_2u+c_2v & b_3u+d_2v & a_2x+b_3y & c_2x+d_2y \\
  0         & b_4u+d_3v & b_4y      & d_3y \\
\end{array}%
\right)$$ and the corresponding characteristic equation
$\det(J(s)-\lambda I)=0.$ The characteristic equation has the form
\begin{equation}\label{A4}
\lambda^{4}-\lambda^{3}p_1+\lambda^{2}p_2+\lambda p_3=0,
\end{equation}
where
\begin{equation*}
\begin{aligned}
p_1 =&a_2x+(b_3+d_3)y+(a_1+b_2)u+d_1v,\\
p_2 =&(a_1b_3+a_1d_3+b_2d_3-b_4d_1-a_2b_1)yu+a_1b_2u^{2}+(a_1d_1-b_1c_1)uv+(b_3d_3-b_4d_2)y^{2} \\
&+(a_2d_3-b_4c_2)xy+(a_2b_2-b_4c_1)xu+(a_2d_1-c_1d_3-a_1c_2)xv+(b_3d_1-b_2d_2-b_1c_2)yv,  \\
p_3 =&a_2d_3xy+b_4c_1xu+(2a_2d_1+c_1d_3-b_3c_1+a_1c_2-b_2c_2)xv+(a_1b_4d_2-a_1b_3d_3)y^{2}u\\
&+(a_1b_4d_1-a_1b_2d_3-b_1b_4c_1)yu^{2}+(a_1b_2d_2-a_1b_3d_1+b_1b_3c_1)yuv\\
&+(a_2b_1c_1+a_1b_2c_2-a_1a_2d_1)xuv+(b_2b_4c_2-a_2b_2d_3-a_1a_2d_3)xyu\\
&+(b_4c_2d_1-b_4c_1d_2+b_3c_1d_3-a_2d_1d_3)xyv+a_2c_1x^{2}(b_4u+d_3v)\\
&+(c_2d_1-c_1d_2)v^{2}(a_1x+b_1y)+b_1c_2d_3y^{2}v.
\end{aligned}
\end{equation*}
Clearly, $\lambda=0$ is the all eigenvalues of the fixed point
$s_1$. Thus $s_1$ is attracting fixed point.

\begin{lemma}\label{A5}
$\lambda=0$ and $\lambda=2$ are eigenvalues for the fixed points
$s_2$, $s_3$, $s_4$, $s_5$.
\end{lemma}
\begin{proof}
From the equation \eqref{A4} it is clear that $\lambda=0$ is an
eigenvalue for all forms of the fixed points. If $s_2=(x,0,u,0)$
is a fixed point with $xu\neq0$ then the coefficients of the
equation \eqref{A4} simplify as
\begin{align*}
p_1&=a_2x+(a_1+b_2)u,\\
p_2&=a_1b_2u^2+(a_2b_2-b_4c_1)xu,\\
p_3&=a_1b_4c_1xu^{2}+a_2b_4c_1x^{2}u.
\end{align*}
Hence we get
$$8-4p_1+2p_2+p_3=(4-2b_2u-b_4c_1xu)(2-a_2x-a_1u)=0$$
which implies that $\lambda=2$ is an eigenvalue.

If $s_3=(0,y,0,v)$ is a fixed point with $yv\neq0$, then the
coefficients of the equation \eqref{A4} simplify as:
\begin{align*}
p_1&=(b_3+d_3)y+d_1v,\\
p_2&=b_3d_3y^2+(b_3d_1-b_1c_2)yv,\\
p_3&=b_1c_2d_1yv^{2}+b_1c_2d_3y^{2}v.
\end{align*}
Consequently,
$$8-4p_1+2p_2+p_3=(4-2b_3y-b_1c_2yv)(2-d_3y-d_1v)=0$$
and thus $\lambda=2$ is an eigenvalue.

If $s_4=(0,y,u,0)$ is a fixed point with $yu\neq0$, then the
coefficients of the equation \eqref{A4} simplify as
\begin{align*}
p_1&=(b_3+d_3)y+(a_1+b_2)u,\\
p_2&=a_1b_2u^2+b_3d_3y^{2}+(a_1b_3+a_1d_3+b_2d_3)yu,\\
p_3&=-a_1b_3d_3y^{2}u-a_1b_2d_3yu^{2}.
\end{align*}
Therefore,
$$8-4p_1+2p_2+p_3=(4-2a_1u-2d_3y-a_1d_3yu)(2-b_3y-b_2u)=0$$
which shows $\lambda=2$ is an eigenvalue.

If $s_5=(x,y,u,v)$ is a fixed point, with $xyuv\neq{0}$, then from
the first and the last equations of the system \eqref{A3} we find
$x=\frac{b_1yu}{1-a_1u}$ and $v=\frac{b_4yu}{1-d_3y}$.
Substituting these values to other equations in the system
\eqref{A3}, we obtain
\begin{equation}\label{A6}
\begin{aligned}
(a_1b_4d_2-a_1b_3d_3)y^2u+b_1c_2d_3y^2v-a_1a_2d_3xyu+a_2d_3xy=
(-a_1b_3-a_1d_3)yu-\\(b_3d_3-b_4d_2)y^2+b_1c_2yv+a_2x+a_1u+
(b_3+d_3)y-1-a_1a_2xu,\\
\end{aligned}
\end{equation}
\begin{equation}\label{A7}
\begin{aligned}
&(a_1b_4d_1-a_1b_2d_3-b_1b_4c_1)yu^2=(a_1+b_2)u+d_3y-1-a_1b_2u^2-\\
&(a_1d_3+b_2d_3-b_4d_1)yu.
\end{aligned}
\end{equation}
Note that we have $u\neq\frac{1}{a_1}$ and $y\neq\frac{1}{d_3}$,
for otherwise the first and the last equations in the system
\eqref{A3} would give us $y=0$ and $u=0$, contradicting to the
condition $xyuv\neq{0}$. Therefore the obtained equations
\eqref{A6} and \eqref{A7} are well defined. Now from the second
and the last equations in the system \eqref{A3} we find
$y=\frac{c_1xv}{1-b_2u-d_1v}$ and $v=\frac{b_4yu}{1-d_3y}$.
Substituting these values into the first, the third and the last
equations in the system \eqref{A3}, we obtain
\begin{align}
&a_1b_2u^2+(a_1d_1-b_1c_1)uv=(a_1+b_2)u+d_1v-1,\\\nonumber
&(b_4c_2d_1-b_4c_1d_2+b_3c_1d_3-a_2d_1d_3)xyv+(b_2b_4c_2-a_2b_2d_3)xyu+
a_2d_3xy+\\
&(a_2d_1-b_3c_1)xv=a_2x+d_3y+b_2u+d_1v-1-b_2d_3yu-a_2b_2xu+b_4c_2xy-d_1d_3yv,\\
&b_4c_1xu+c_1d_3xv=1-b_2u-d_1v.
\end{align}
Moreover, from the first and the third equations in the system
\eqref{A3} we find $x=\frac{b_1yu}{1-a_1u}$ and
$u=\frac{c_2xv+d_2yv}{1-a_2x-b_3y}$. Substituting these values
into the second and the last equations in the system \eqref{A3} we
obtain
\begin{equation}
\begin{aligned}
&(a_1b_2d_2-a_1b_3d_1+b_1b_3c_1)yuv+(a_2b_1c_1+a_1b_2c_2-a_1a_2d_1)xuv+\\
&(a_2d_1-b_2c_2)xv=a_2x+b_3y+a_1u+d_1v-1-
a_1b_3yu-\\
&(a_1d_1-b_1c_1)uv-(b_3d_1-b_2d_2)yv-a_1a_2xu,
\end{aligned}
\end{equation}
\begin{equation}
(b_3d_3-b_4d_2)y^2+(a_2d_3-b_4c_2)xy=a_2x+(b_3+d_3)y-1.
\end{equation}
Taking into account all the obtained equations and the system
\eqref{A3} we get
\begin{equation}\label{A8}
\begin{aligned}
8-4p_1+2p_2+p_3=&
%1-b_3y-(b_4d_1+2a_2b_1)yu+(a_2b_2-2a_1a_2)xu
%+(2a_2d_1-a_1c_2)xv+(b_3d_1-b_2d_2-b_1c_2-d_1d_3)yv
%+a_2c_1x^2(b_4u+d_3v)+(c_2d_1-c_1d_2)v^2(a_1x+b_1y)
%=1-2a_2x-b_3y-d_1v+a_2b_2xu+(2a_2d_1-a_1c_2)xv
%+(b_3d_1-b_2d_2-b_1c_2)yv+a_2c_1x^2\frac{v}{y}
%+(c_2d_1-c_1d_2)v^2\frac{x}{u}
\frac{1-d_1v}{u}(u-a_2xu-c_2xv
-b_3yu-d_2yv)\\
&+\bigl(\frac{d_2v}{u}-\frac{a_2x}{y}\bigr)
(y-c_1xv-b_2yu-d_1yv)\\
&+\frac{c_2v}{u}(x-a_1xu-b_1yu)=0.
\end{aligned}
\end{equation}
This shows that $\lambda=2$ is an eigenvalue for nonzero fixed
point of operator \eqref{A1}. Lemma~\ref{A5} is proved.
\end{proof}

\begin{remark}
We have proved the Lemma \ref{A5} for the case when
$b_1,b_4,c_1,c_2x+d_2y$ are nonzero. In case when some of the
numbers $b_1,b_4,c_1,c_2x+d_2y$ are zero, then
Lemma~{\upshape\ref{A5}} can be proven similarly. For instance if
we have only $b_4=0$ then the last equation of the system
\eqref{A3} gives us $y=\frac{1}{d_3}$ as $xyuv\neq{0}$. In this
case we can rewrite \eqref{A8} as
\begin{align*}
8-4p_1+2p_2+p_3=&\frac{1-d_1v}{u}(u-a_2xu-c_2xv-b_3yu-d_2yv)\\
&+\bigl(\frac{d_2v}{u}-\frac{a_2x}{y}+\frac{1}{y}\bigr)
(y-c_1xv-b_2yu-d_1yv)\\
&+\frac{c_2v}{u}(x-a_1xu-b_1yu)=0.
\end{align*}
\end{remark}
\begin{conjecture}
$\lambda=0$ and $\lambda=2$ are eigenvalues of the gonosomal
evolution operator \eqref{A2} corresponding to nonzero fixed
points.
\end{conjecture}

\begin{lemma}
If either $p_1-p_2=3$ or $3p_1-p_2=7$ holds, then $s=(x,y,u,v)$ is
a nonhyperbolic fixed point. Otherwise it is a saddle point in the
case $x^2+y^2+u^2+v^2\neq{0}$.
\end{lemma}
\begin{proof}
For other roots of the equation \eqref{A4}  we have
\begin{align*}
\lambda_{3,4}&=-1+\frac{p_1\pm\sqrt{p_1^2+4p_1-4p_2-12}}{2}\\
&=-1+\frac{p_1\pm\sqrt{p_1^2+4(p_1-p_2-3)}}{2}\\
&=-1+\frac{p_1\pm\sqrt{(p_1-4)^2+4(3p_1-p_2-7)}}{2},
\end{align*}
which completes the proof.
\end{proof}
\begin{corollary} It holds that
\begin{align*}
&s_2 \quad is \quad \left\{
\begin{aligned}
\begin{array}{ll}
saddle      &\quad if \quad  b_4c_1\neq{a_2(a_1\pm{b_2})}, \cr
nonhyperbolic   &\quad if \quad  b_4c_1={a_2(a_1\pm{b_2})},\\
\end{array}
\end{aligned}
\right.\\[0,5ex]
&s_3 \quad is \quad \left\{
\begin{aligned}
\begin{array}{ll}
saddle      &\quad if \quad b_1c_2\neq{d_1(d_3\pm{b_3})}, \cr
nonhyperbolic &\quad if \quad b_1c_2={d_1(d_3\pm{b_3})},
\end{array}
\end{aligned}
\right.\\[0,5ex]
&s_4 \quad is \quad \left\{
\begin{aligned}
\begin{array}{ll}
saddle    &\quad if \quad a_1d_3\neq-1\pm(a_1b_3+b_2d_3), \cr
nonhyperbolic &\quad if \quad a_1d_3=-1\pm{(a_1b_3+b_2d_3)}.
\end{array}
\end{aligned}
\right.
\end{align*}
\end{corollary}
\begin{corollary}
Let $s=(x,y,u,v)$ be a fixed point for the operator \eqref{A1}
such that $x^2+y^2+u^2+v^2\neq{0}$. Then it is either
nonhyperbolic or saddle point. Furthermore the gonosomal evolution
operator \eqref{A1} does not have repelling fixed points.
\end{corollary}
%================================================================
\section{The $\omega$-limit set and the main results.}\label{s2}
\noindent The problem of describing the $\omega$-limit set of a
trajectory is of great importance in the theory of dynamical
systems.
\begin{proposition}\label{A9}
The point $s=(0,0,...,0)\in \mathbb{R}^{\eta+\nu}$ is a fixed
point for the operator \eqref{A2}. If $\delta\in [0,4)$ and the
coefficients of the operator \eqref{A2} are nonnegative real
numbers, then for any initial point $t\in Q_\delta$, we have
\begin{equation}\label{Int.-6}
\lim\limits_{n\rightarrow \infty}W^n (t)=\underbrace
{(0,0,...,0)}_{\eta+\nu} \ ,
\end{equation}
where $$Q_\delta=\{(x_1,...,x_\eta,y_1,...,y_\nu)\in
\mathbb{R}^{\eta+\nu}:
\sum_{j=1}^{\eta}x_j+\sum_{l=1}^{\nu}y_l\leq{\delta}, \,
x_j\geq{0}, y_l\geq{0}, j=\overline{1,\eta},
l=\overline{1,\nu}\}$$
\end{proposition}
\begin{proof} It is not difficult to see that
$s=(0,0,...,0)\in \mathbb{R}^{\eta+\nu}$ is an attracting fixed
point for the operator \eqref{A2}. If $t\in{Q_\delta}$, then from
\eqref{A2} we get $x'_j\geq{0}$, $y'_l\geq{0}$, for
$j=\overline{1,\eta}$, $l=\overline{1,\nu}$, and
$$\sum_{j=1}^{\eta}x'_j+\sum_{l=1}^{\nu}y'_l=
\sum_{j=1}^{\eta}x_j\cdot\sum_{l=1}^{\nu}y_l\leq\frac{1}{4}\Bigl(\sum_{j=1}^{\eta}x_j+
\sum_{l=1}^{\nu}y_l\Bigr)^2\leq{\frac{\delta^2}{4}<\delta}.$$
Therefore $$t'=(x'_1,...,x'_\eta,y'_1,...,y'_\nu)\in
Q_{\frac{\delta^2}{4}}\subset {Q_\delta}.$$ Denoting
$f(\delta)=\frac{\delta^2}{4}$, we can write
$$W^{n}(Q_\delta)\subset{W^{n-1}(Q_{f(\delta)})}\subset{W^{n-2}(Q_{f^{2}(\delta)})}
\subset...\subset{Q_{f^{n}(\delta)}}.$$ Since
$\lim\limits_{n\rightarrow{\infty}}f^{n}(\delta)=
\lim\limits_{n\rightarrow{\infty}}4\bigl(\frac{\delta}{4}\bigr)^{2^n}=0$,
we get $$\lim\limits_{n\rightarrow \infty}W^n
(Q_\delta)\subset{Q_0}=\{{\underbrace {(0,0,...,0)}_{\eta+\nu}}\}
\ ,$$ which completes the proof.
\end{proof}
If $a_1\in{(0,1)}$, then for any initial point $t_0=(x_0,0,u_0,0)$
for the operator \eqref{A1} we have
\begin{align*}
\lim\limits_{n\rightarrow\infty}W^{n}(t_0)
&=\Bigl(\frac{1}{a_2}(a_1a_2x_0u_0)^{2^{n-1}},
0,\frac{1}{a_1}(a_1a_2x_0u_0)^{2^{n-1}},0\Bigr)\\[1ex]
&= \left\{
\begin{array}{ll}
(0,0,0,0), &\quad if \quad |x_0u_0|<\frac{1}{a_1a_2},\\
\bigl(\frac{1}{a_2},0,\frac{1}{a_1},0\bigr), &\quad if \quad |x_0u_0|=\frac{1}{a_1a_2},\\
+\infty, &\quad if \quad |x_0u_0|>\frac{1}{a_1a_2}.
\end{array}
\right.
\end{align*}

If  $a_1=0$, then $W(t_0 )=(0,0,x_0 u_0,0)$ and $W^{n}(t_0
)=(0,0,0,0)$ for all $n\geq{2}$. If  $a_1=1$, then $W(t_0 )=(x_0
u_0,0,0,0)$ and $W^{n}(t_0 )=(0,0,0,0)$ for all $n\geq{2}$. Thus
for the cases $a_1=0$ and $a_1=1$, we have
$$\lim\limits_{n\rightarrow \infty}W^{n}(t_0 )=(0,0,0,0).$$
If $d_1,d_3\in{(0,1)}$ and $d_2=0$, then for any initial point
$t_0=(0,y_0,0,v_0)$ we have
\begin{align*}
\lim\limits_{n\rightarrow
\infty}W^{n}(t_0)&=\Bigl(0,\frac{1}{d_3}(d_1d_3y_0v_0)^{2^{n-1}},
0,\frac{1}{d_1}(d_1d_3y_0v_0)^{2^{n-1}}\Bigr)\\[1ex]
&=\left\{
\begin{array}{ll}
(0,0,0,0), &\quad  if \quad |y_0v_0|<\frac{1}{d_1d_3}, \\
\bigl(0,\frac{1}{d_3},0,\frac{1}{d_1}\bigr), &\quad if \quad |y_0v_0|=\frac{1}{d_1d_3}, \\
+\infty, &\quad if \quad |y_0v_0|>\frac{1}{d_1d_3}.
\end{array}
\right.
\end{align*}
If $d_1=d_2=0$, then $W(t_0)=(0,0,0,y_0v_0)$ and $W^{n}(t_0
)=(0,0,0,0)$ for all $n\geq{2}$. If  $d_2=d_3=0$, then $W(t_0
)=(0,y_0v_0,0,0)$ and $W^{n}(t_0 )=(0,0,0,0)$ for all $n\geq{2}$.
Thus for the cases $d_1=d_2=0$ and $d_2=d_3=0$, we have
$$\lim\limits_{n\rightarrow \infty}W^{n}(t_0 )=(0,0,0,0).$$
If $b_2,b_3\in{(0,1)}$ and $b_1=b_4=0$, then for any initial point
$t_0=(0,y_0,u_0,0)$ we have
\begin{align*}
\lim\limits_{n\rightarrow\infty}W^{n}(t_0)
&=\Bigl(0,\frac{1}{b_3}(b_2b_3y_0u_0)^{2^{n-1}},
\frac{1}{b_2}(b_2b_3y_0u_0)^{2^{n-1}},0\Bigr)\\[1ex]
&= \left\{
\begin{array}{ll}
(0,0,0,0), &\quad if \quad |y_0u_0|<\frac{1}{b_2b_3}, \\
\bigl(0,\frac{1}{b_3},\frac{1}{b_2},0\bigr), &\quad if \quad |y_0u_0|=\frac{1}{b_2b_3}, \\
+\infty, &\quad if \quad |y_0u_0|>\frac{1}{b_2b_3}.
\end{array}
\right.
\end{align*}
If $b_1=b_2=b_4=0$, then $W(t_0)=(0,0,y_0u_0,0)$ and $W^{n}(t_0
)=(0,0,0,0)$ for all $n\geq{2}$. If $b_1=b_3=b_4=0$, then $W(t_0
)=(0,y_0u_0,0,0)$ and $W^{n}(t_0 )=(0,0,0,0)$ for all $n\geq{2}$.
Hence for cases $b_1=b_2=b_4=0$ and $b_1=b_3=b_4=0$ we have
$$\lim\limits_{n\rightarrow \infty}W^n (t_0 )=(0,0,0,0)$$

\begin{lemma}\label{B1}
Let
$$Q_4=\{(x,y,u,v)\in{\mathbb{R}^4}: x\geq0,
y\geq0, u\geq0,  v\geq0, \, x+y+u+v\leq{4} \}.$$ For any initial
point $t\in{Q_4}$ if there exists $k\geq0$ such that
$$\bigl(a_1-\frac{1}{2}\bigr)x^{(k)}u^{(k)}+\bigl(c_1-\frac{1}{2}\bigr)x^{(k)}v^{(k)}
+\bigl(b_1+b_2-\frac{1}{2}\bigr)y^{(k)}u^{(k)}+\bigl(d_1-\frac{1}{2}\bigr)y^{(k)}v^{(k)}\neq{0}$$
then
\begin{equation}\label{B2}
\lim\limits_{n\rightarrow \infty}W^n (t)=s_1=(0,0,0,0).
\end{equation}
\end{lemma}
\begin{proof} Since $t\in{Q_4}$ and
$$x'+y'+u'+v'=(x+y)(u+v)\leq{\bigl(\frac{x+y+u+v}{2}\bigr)^2}=4,$$ we have
$x+y=2$, $u+v=2$. Otherwise, $x'+y'+u'+v'<4$ and \eqref{B2}
follows by Proposition~\ref{A9}. Hence
$$x'+y'+u'+v'=4,$$ where
$$x'+y'=2+\bigl(a_1-\frac{1}{2}\bigr)xu+\bigl(c_1-\frac{1}{2}\bigr)xv+
\bigl(b_1+b_2-\frac{1}{2}\bigr)yu+\bigl(d_1-\frac{1}{2}\bigr)yv$$
and
$$u'+v'=2-[\bigl(a_1-\frac{1}{2}\bigr)xu+\bigl(c_1-\frac{1}{2}\bigr)xv+
\bigl(b_1+b_2-\frac{1}{2}\bigr)yu+\bigl(d_1-\frac{1}{2})yv\bigr].$$
If $\bigl(a_1-\frac{1}{2}\bigr)xu+\bigl(c_1-\frac{1}{2}\bigr)xv+
\bigl(b_1+b_2-\frac{1}{2}\bigr)yu+\bigl(d_1-\frac{1}{2}\bigr)yv\neq0$,
then
$$W^{2}(t)\in{Q_{4-\bigl[\bigl(a_1-\frac{1}{2}\bigr)xu+\bigl(c_1-\frac{1}{2}\bigr)xv+
\bigl(b_1+b_2-\frac{1}{2}\bigr)yu+\bigl(d_1-\frac{1}{2}\bigl)yv\bigl]^2}}$$
and \eqref{B2} again follows by Proposition~\ref{A9}. Repeating
this argument we get that, if
$$\bigl(a_1-\frac{1}{2}\bigr)x'u'+\bigl(c_1-\frac{1}{2}\bigr)x'v'+
\bigl(b_1+b_2-\frac{1}{2}\bigr)y'u'+\bigl(d_1-\frac{1}{2}\bigl)y'v'\neq0,$$
then $$W^{3}(t)\in{Q_{4-\bigl[\bigl(a_1-\frac{1}{2}\bigr)x'u'+
\bigl(c_1-\frac{1}{2}\bigr)x'v'+\bigl(b_1+b_2-\frac{1}{2}\bigr)y'u'+
\bigl(d_1-\frac{1}{2}\bigr)y'v'\bigr]^2}}.$$ Otherwise, we iterate
the argument again and conclude that if there exists $k\geq0$ such
that $$\bigl(a_1-\frac{1}{2}\bigr)x^{(k)}u^{(k)}+
\bigl(c_1-\frac{1}{2}\bigr)x^{(k)}v^{(k)}+\bigl(b_1+b_2-\frac{1}{2}\bigr)y^{(k)}u^{(k)}+
\bigl(d_1-\frac{1}{2}\bigr)y^{(k)}v^{(k)}\neq{0},$$ then
\eqref{B2} is satisfied.
\end{proof}
\begin{lemma}\label{B3}
Let $$\Delta=\{(x,y,u,v)\in{\mathbb{R}^4}: x\geq0, y\geq0, u\geq0,
v\geq0, \, x+y+u+v>4\}.$$ For any initial point $t\in{\Delta}$,

\smallskip
\noindent {\rm(i)} if there exists $k\geq0$ such that $(x^{(k)}
+y^{(k)})(u^{(k)}+v^{(k)})<4$, then {\upshape{\eqref{B2}}} holds.

\smallskip
\noindent {\rm(ii)} if \, $\max\{a_1a_2xu, b_2b_3yu,
d_1d_3yv\}>1$, then $\lim\limits_{n\rightarrow
\infty}W^{n}(t)=\infty$, i.e. at least one coordinate of
$W^{n}(t)$ tends to $\infty$ as $n\to\infty$.
\end{lemma}
\begin{proof}
Part (i) of this lemma simply follows from the identity
$$x^{(k+1)} +y^{(k+1)}+u^{(k+1)}+v^{(k+1)}=(x^{(k)} +y^{(k)})(u^{(k)}+v^{(k)}),$$
and by the Proposition \ref{A9}. We prove the claim in (ii) for
the case
$$\text{max}\{a_1a_2xu, \, b_2b_3yu, \,
d_1d_3yv\}=a_1a_2xu>1.$$ Other cases can be proven similarly. To
this end, observe that for any $t=(x,y,u,v)\in{\mathbb{R}^4}$ with
$x\geq0, y\geq0, u\geq0, v\geq0\}$ we get from \eqref{A1} that
$$x^{(k+1)}\geq{a_1x^{(k)}u^{(k)}}, \quad
u^{(k+1)}\geq{a_2x^{(k)}u^{(k)}}, \quad k=0,1,... \ .$$ By
iterating these inequalities, we obtain
$$x^{(k+1)}\geq{\frac{1}{a_2}{[a_1a_2xu]^{2^k}}}, \quad
u^{(k+1)}\geq{\frac{1}{a_1}{[a_1a_2xu]^{2^k}}}, \quad k=0,1,... \
.$$ This completes the proof.
\end{proof}
Part (ii) of Lemma  \ref{B3}  can be generalized as follows.
\begin{proposition}\label{D}
Let the coefficients of the operator \eqref{A2} and the
coordinates of an initial point $t$ be nonnegative real numbers.
If
\begin{equation*}
\max_{\substack{1\leq i,\,r\leq \eta, \\ 1\leq
j,\,l\leq\nu}}\bigl\{p_{ij,r}^{(f)}p_{ij,l}^{(m)}x_{r}y_l\bigr\}>1,
\end{equation*}
then $\lim\limits_{n\rightarrow \infty}W^{n}(t)=\infty$, i.e. at
least one coordinate of $W^{n}(t)$ tends to $\infty$ as
$n\to\infty$.
\end{proposition}
\begin{proof}
We prove the claim only for the case
\begin{equation*}
\max_{\substack{1\leq i,\,r\leq \eta,\\ 1\leq
j,\,l\leq\nu}}\bigl\{p_{ij,r}^{(f)}p_{ij,l}^{(m)}x_{r}y_l\bigr\}
=p_{11,1}^{(f)}p_{11,1}^{(m)}x_{1}y_1>1.
\end{equation*}
Other cases can be proven similarly. To this end, observe that for
any initial point $t\in P$, we get from \eqref{A2} that
$${x_{1}}^{(k+1)}\geq{p_{11,1}^{(f)}{x_1}^{(k)}{y_1}^{(k)}}, \quad
{y_1}^{(k+1)}\geq{p_{11,1}^{(m)}{x_1}^{(k)}{y_1}^{(k)}}, \quad
k=0,1,... \ .$$ By iterating these inequalities, we obtain
$${x_{1}}^{(k+1)}\geq{\frac{1}{p_{11,1}^{(m)}}
{[p_{11,1}^{(f)}p_{11,1}^{(m)}x_{1}y_1]^{2^k}}}, \quad
{y_{1}}^{(k+1)}\geq{\frac{1}{p_{11,1}^{(f)}}{[p_{11,1}^{(f)}
p_{11,1}^{(m)}x_{1}y_1]^{2^k}}}, \quad k=0,1,.. \ .$$ If
$p_{11,1}^{(f)}p_{11,1}^{(m)}=0$ then we go to the other cases.
Proposition \ref{D} is proved.
\end{proof}
Let us make the notations
\begin{align*}
O        &=\{(0,0,u,v)\in\mathbb{R}^{4}:u,v\in\mathbb{R}\}\cup\{(x,y,0,0)\in\mathbb{R}^{4}:x,y\in\mathbb{R}\}\\
I        &=\{(x,y,u,v)\in\mathbb{R}^{4}:y=v=0 \}\\
J        &=\{(x,y,u,v)\in{I}: x=u\}\\
P        &=\{(x,y,u,v)\in\mathbb{R}^{4}: x\geq0, y\geq0, u\geq0, v\geq0 \}\\
P_0      &=\{(x,y,u,v)\in{P}: (x+y)(u+v)<4\}\\
Q_a      &=\{(x,y,u,v)\in{P}: x+y+u+v\leq{a}\}, \quad a\in{[0,4]}\\
N        &=\{(x,y,u,v)\in\mathbb{R}^{4}: x\leq0, y\leq0, u\leq0, v\leq0 \}\\
N_0      &=\{(x,y,u,v)\in\mathbb{R}^{4}: x\leq0, y\leq0, u\geq0, v\geq0 \}\\
N_1      &=\{(x,y,u,v)\in\mathbb{R}^{4}: x\geq0, y\geq0, u\leq0, v\leq0 \}\\
\Delta_0 &=\{(x,y,u,v,)\in{P}: x+y+u+v>4, \quad max\{a_1a_2xu, \,
b_2b_3yu, \, d_1d_3yv\}>1\}.
\end{align*}
The sets $I$, $J$, $P$ and $Q_a$, where $a\in{[0,4]}$, are
invariant with respect to the operator \eqref{A1}. Moreover, we
have
$$W(O)=\{(0,0,0,0)\}, \quad W(Q_a)\subset{Q_\frac{a^2}{4}}, \quad W(N)\subset{P}, \quad W(N_0)\subset{N}, \quad
W(N_1)\subset{N}.$$ Summarizing above observations, we get the
following result.

\begin{theorem}\label{thm:main}
If $t=(x,y,u,v)\in{\mathbb{R}^4}$ is such that

\smallskip
\noindent {\rm(i)} one of the following conditions is satisfied

1) $t\in{P_0}$,

2) $t\in{Q_4}$ and {\upshape{Lemma \ref{B1}}} holds,

3) $t\in{N},  \,\, \quad W(t)\,\;\in{P_0}$,

4) $t\in{N_0}, \quad W^{2}(t)\in{P_0}$,

5) $t\in{N_1}, \quad W^{2}(t)\in{P_0}$,

then $$\lim\limits_{n\rightarrow \infty}W^n (t)=s_1=(0,0,0,0).$$

\noindent {\rm(ii)} one of the following conditions is satisfied

1) $t\in{\Delta_0}$,

2) $t\in{N},   \, \, \quad W(t) \,\; \in{\Delta_0}$,

4) $t\in{N_0}, \quad W^{2}(t)\in{\Delta_0}$,

5) $t\in{N_1}, \quad W^{2}(t)\in{\Delta_0}$,

then $$\lim\limits_{n\rightarrow \infty}W^n (t)=\infty,$$ i.e. at
least one coordinate of $W^{n}(t)$ tends to $\infty$.
\end{theorem}

\section{Conclusion}
\noindent We have considered the dynamical systems of a hemophilia
generated by gonosomal evolution operator of sex linked
inheritance in $\mathbb{R}^4$ depending on parameters and studied
their trajectory behavior. In Section \ref{S1} it is proven that
operator \eqref{A1} has a unique attracting fixed point and the
other fixed points might be either nonhyperbolic or saddle. We
note that the union of sets for initial points considered in
Theorem~{\upshape{\ref{thm:main}}} does not cover $\mathbb{R}^4$
and the question of description of the entire $\omega$-limit sets
for the fixed points $s_2, s_3, s_4, s_5$ is remained as an open
problem. However, due to the eigenvalues which we have found in
section \ref{S1}, we can give exact measure of stable and unstable
manifolds of those fixed points, see \cite{Devaney.book.2003} for
more details. The dynamical systems considered in this paper are
interesting as they are examples for nonlinear higher dimensional
discrete-time dynamical systems that have not been fully
understood yet.


\begin{thebibliography}{999}
\bibitem{Devaney.book.2003}
R. L. Devaney,  {\it An introduction to chaotic dynamical system},
Westview Press, 2003.

\bibitem {Ganikhodzhaev.book.2011}
R. N. Ganikhodzhaev, F.M. Mukhamedov, U.A. Rozikov,  {\it
Quadratic stochastic operators and processes: results and open
problems}. Inf. Dim. Anal. Quant. Prob. Rel. Fields., 14(2)
(2011), 279-335.

\bibitem {Karlin.book.1984}
S. Karlin, {\it Mathematical models, problems and controversics of
evolutionary theory}. Bull. Amer. Math. Soc. (N.S) 10(2) (1984),
221-274.

\bibitem {Ladra.book.2013}
M. Ladra, U.A. Rozikov,  {\it Evolution algebra of a bisexual
population}, Jour. Algebra. 378(2013), 153-172.

\bibitem {Lyubich.book.1992}
Y. I. Lyubich,  {\it Mathematical structures in population
genetics}, Springer-Vergar, Berlin 1992.

\bibitem {Reed.book.1997}
M. L. Reed, {\it Algebraic structure of genetic inheritance},
Bull. Amer. Math. Soc. (N.S.) 34 (2) (1997) 107-130.

\bibitem{Rozikov.book.2013}
U. A. Rozikov,  {\it Evolution operators and algebras of sex
linked inheritance}. Asia Pacific Math. Newsletter. 3 (1) (2013),
6-11.

\bibitem{Rozikov.book.2011}
U. A. Rozikov, U.U. Zhamilov,  {\it Volterra quadratic stochastic
operators of bisexual population}. Ukraine Math. Jour., 63(7)
(2011), 985-998.

\bibitem{Rozikov.book.2016}
U. A. Rozikov, R. Varro,  {\it Dynamical systems generated by a
gonosomal evolution operator}, Discontinuity, Nonlinearity and
Complexity 2016, V.5,  p. 173-185.

\bibitem {Varro.book.2015}
R. Varro, {\it Gonosomal algebra}, 2015, Jour. Algebra, 447
(2016), p. 1-30.

\end{thebibliography}
\end{document}